\documentclass[11pt,a4paper]{amsart}
\usepackage{graphicx}
\usepackage{amssymb}
\usepackage[all]{xy}
\usepackage{amsmath}
\usepackage[french,english]{babel}
\usepackage{amssymb,amsmath,amsfonts,amsthm,amscd}
\usepackage{indentfirst,graphicx}
\usepackage[font={scriptsize},captionskip=5pt]{subfig}
\def\pd#1#2{\frac{\partial#1}{\partial#2}}
\def\Projan{\mathop{\rm Projan}}
\newtheorem{theorem}{Theorem}[section]
\newtheorem{exa}[theorem]{Example}

\newtheorem{rem}[theorem]{Remark}

\newtheorem{cor}[theorem]{Corollary}

\newtheorem{tont}[theorem]{Definition}
\newenvironment{definition}{\begin{tont} \em}{\end{tont}}
\newtheorem{lem}[theorem]{Lemma}
\newenvironment{Lemme}{\begin{lem} \em}{\end{lem}}

\newtheorem{Prop}[theorem]{Proposition}
 
\let\:=\colon

\newcommand{\cO}{{\mathcal O}}

\newcommand{\cR}{{\mathcal R}}

\begin{document}

\title [Infinitesimal Lipschitz condition ]{The genericity of the infinitesimal Lipschitz condition for hypersurfaces}
\author{TERENCE GAFFNEY}

\maketitle

\selectlanguage{english}

\begin{abstract}
We continue the development of the theory of infinitesimal Lipschitz equivalence, showing the genericity of the condition for families of hypersurfaces with isolated singularities.
\end{abstract}

\section{Introduction} In an earlier paper \cite{GL1}, we introduced a candidate for a theory of infinitesimal Lipschitz equisingularity for families of complex analytic hypersurfaces with isolated singularities. The definition given there has an equivalent formulation, using the theory of integral closure of modules. This alternate form is easier to work with in many situations. In this paper we show that a slightly evolved version of this condition is  {\it generic}. More precisely, we show, in the case of two strata, considered here, that the condition holds on a Zariski open subset of the parameter stratum $Y$. Proving that a stratification property is generic is essential for an equisingularity condition to have any value.

In preparation for using the integral closure formulation of our condition, we review some elements  of the theory of integral closure of modules in section 2.

In section 3, we review the definition of the Lipschitz saturation of an ideal, give its alternate formulation using the theory of integral closure and define two infinitesimal Lipschitz conditions, one which we denote by $iL_{m_Y}$ which is the analogue of the Whitney conditions and one which is the analogue of the Whitney A or the a$_f $ condition which we denote by $iL_A$.  We also give a geometric interpretation of these conditions on the family $X$.

We also introduce an invariant coming from the integral formulation of the Lipschitz condition. We use this invariant to show when two different ideals have the same Lipschitz saturation. We also use it to characterize generic hyperplanes in section 4.

In section 4, we come to the heart of this paper. As mentioned earlier, proving a genericity theorem is an important step in developing the theory attached to an equisingularity condition. Not only is this result necessary to ensure the condition is widely applicable, but the fact of genericity implies a strong connection with the geometry of the family. For example, Teissier proved that condition C held on a Zariski open and dense subset of the parameter space $Y^k$, of a $k$ parameter family of isolated hypersurface singularities in ${\Bbb C}^{n+k}$ in \cite{T1}. Condition C later was seen to be equivalent to Verdier's condition W$_f$ for the pair of strata $\{{\Bbb C}^{n+k}-Y^k, Y^k\}$, where $f$ defined the family. Condition C was the keystone of Teissier's work on the Whitney equisingularity of families of hypersurfaces with isolated singularities. We use Teissier's proof in \cite{T1} as  a model in developing a similar theorem for the $iL_A$ condition. Currently a proof for the genericity of the $iL_{m_Y}$ remains unknown.

In section 4, we state and prove the genericity theorem for the $iL_A$ condition for the case of families of isolated hypersurface singularities.  For the proof, we work in the module setting.  Analogous results exist in the general case for families of isolated singularities, but requires further work in developing the definition of the infinitesimal Lipschitz condition; since you start with modules in the general case instead of ideals, a further layer of complexity is added in passing to the module theoretic version of the definition. 

Also in section 4, we give an application of the genericity theorem. Given an equisingularity condition it is natural to ask if it passes to the family of generic plane sections of the singularity. We use the genericity theorem to show that it does for the $iL_A$ condition. We then use the invariant introduced in section 3, and the multiplicity polar theorem, discussed in section 2, to give a condition for a hyperplane to be generic.

Ultimately, we hope to use the stratification condition defined here to prove that for a family of isolated hypersurface singularities,  the $iL_A$ condition gives a necessary and sufficient condition for the family to have a bi-Lipschitz stratification which includes $Y$ as a stratum. This would give an infinitesimal criterion for the existence of a bi-Lipschitz stratification of such a family. It is known by work of Mostowski, \cite{M} that bi-Lipschitz stratifications exist in the complex analytic setting, but not much is known about them besides their existence.

Using the conditions of this paper to characterize the ``thick" and ``thin" zones of Birbrair, Neumann and Pichon \cite{BNP}, developed by them for normal surface singularities, would open an avenue to generalizing these notions to higher dimensions, as well as linking them with Mostowski's work on showing the existence of these stratifications. 

I am happy to acknowledge the impetus to this work given by the beautiful paper of Birbrair,  Neumann and  Pichon \cite {BNP} and the stimulation afforded from conversation with them.

\vspace{1cm}

\section{The theory of the Integral closure of modules}

Let  $(X, x)$ be a germ of a complex analytic space and $X$ a
small representative of the germ and let $\mathcal{O}_{X}$ denote the
structure sheaf on a complex analytic space $X$. One of the formulations of the definition of the infinitesimal Lipschitz condition uses the theory of
integral closure of modules, which we now review. This theory will also provide the tools for working with the condition.

\begin{definition} Suppose $(X, x)$ is the germ of a complex analytic space,
$M$ a submodule of $\mathcal{O}_{X,x}^{p}$. Then $h \in
\mathcal{O}_{X,x}^{p}$ is in the integral closure of $M$, denoted
$\overline{M}$, if for all analytic $\phi : (\mathbb{C}, 0) \to (X,
x)$, $h \circ \phi \in (\phi^{*}M)\mathcal{O}_{1}$. If $M$ is a
submodule of $N$ and $\overline{M} = \overline{N}$ we say that $M$
is a reduction of $N$.
\end{definition}

To check the definition it suffices to check along a finite number of curves whose generic point is in the Zariski open subset of $X$ along which $M$ has maximal rank. (Cf. \cite {G-2}.)

If a module $M$ has finite colength in $\mathcal{O}_{X,x}^{p}$, it
is possible to attach a number to the module, its Buchsbaum-Rim
multiplicity,  $e(M,\mathcal{O}_{X,x}^{p}).$ We can also define the multiplicity $e(M,N)$ of a pair of
modules $M \subset N$, $M$ of finite colength in $N$, as well, even
if $N$ does not have finite colength in $\mathcal{O}_{X}^{p}$.

We recall how to construct the multiplicity of a pair of modules using the approach of
Kleiman and Thorup \cite{KT}. Given a submodule $M$ of a free
$\mathcal{O}_{X^{d}}$ module $F$ of rank $p$, we can associate a
subalgebra $\mathcal{R}(M)$ of the symmetric $\mathcal{O}_{X^{d}}$
algebra on $p$ generators. This is known as the Rees algebra of $M$.
If $(m_{1}, \cdots ,m_{p})$ is an element of $M$ then $\sum
m_{i}T_{i}$ is the corresponding element of $\mathcal{R}(M)$. Then
$\Projan(\mathcal{R}(M))$, the projective analytic spectrum of
$\mathcal{R}(M)$ is the closure of the projectivised row spaces of
$M$ at points where the rank of a matrix of generators of $M$ is
maximal. Denote the projection to $X^{d}$ by $c$. If $M$ is a
submodule of $N$ or $h$ is a section of $N$, then $h$ and $M$
generate ideals on $\Projan \mathcal{R}(N)$; denote them by $\rho(h)$
and $\rho(\mathcal{M})$. If we can express $h$ in terms of a set of
generators $\{n_{i}\}$ of $N$ as $\sum g_{i}n_{i},$ then in the
chart in which $T_{1}\neq 0,$ we can express a generator of
$\rho(h)$ by $\sum g_{i}T_{i}/T_{1}.$ Having defined the ideal sheaf
$\rho(\mathcal{M}),$ we blow it up.

On the blow up $B_{\rho(\mathcal{M})}(\Projan \mathcal{R}(N))$ we have
two tautological bundles. One is the pullback of the bundle on
$\Projan \mathcal{R}(N)$. The other comes from $\Projan
\mathcal{R}(M)$. Denote the corresponding Chern classes by $c_{M}$
and $c_{N}$, and denote the exceptional divisor by $D_{M,N}$.
Suppose the generic rank of $N$ (and hence of $M$) is $g$.

Then the multiplicity of a pair of modules $M, N$ is:

$$
e(M,N) = \sum_{j=0}^{d+g-2}\int D_{M,N}\cdot c_{M}^{d+g-2-j}\cdot
c_{N}^{j}.
$$

Kleiman and Thorup show that this multiplicity is well defined at $x
\in X$ as long as $\overline{M} = \overline{N}$ on a deleted
neighborhood of $x$. This condition implies that $D_{M,N}$ lies in
the fiber over $x$, hence is compact. Notice that when $N=F$ and $M$ has finite colength in $F$ then $e(M,N)$ is the Buchsbaum-Rim multiplicity $e(M,\mathcal{O}_{X,x}^{p})$. There is a fundamental result due to Kleiman and Thorup, the principle of additivity \cite{KT}, which states that given a sequence of $\mathcal{O}_{X,x}$-modules $M\subset N \subset P$  such that the multiplicity of the pairs is well defined, then$$e(M,P)=e(M,N)+e(N,P).$$Also if $\overline{M}=\overline{N}$ then $e(M,N)=0$ and the converse also holds if $X$ is equidimensional. Combining these two results we get thet if $\overline{M}=\overline{N}$ then $e(M,N)=e(N,P).$ These results will be used in Section 5.

In studying the geometry of singular spaces, it is natural to study
pairs of modules. In dealing with non-isolated singularities, the
modules that describe the geometry have non-finite colength, so
their multiplicity is not defined. Instead, it is possible to define
a decreasing sequence of modules, each with finite colength inside
its predecessor, when restricted to a suitable complementary plane.
Each pair controls the geometry in a particular codimension.

We also need the notion of the polar varieties of $M$. The {\it polar variety of codimension $k$} of $M$ in $X$, denoted
$\Gamma_k(M)$, is constructed by intersecting $\Projan{\mathcal R}(M)$
 with $X\times H_{g+k-1}$ where
$H_{g+k-1}$ is a general plane of codimension $g+k-1$, then projecting to
$X$.

Setup: We suppose we have families  of modules $M\subset  N$, $M$ and $N$
 submodules of a free module $F$ of rank $p$
 on an equidimensional family of spaces with equidimensional
 fibers ${\mathcal X}^{d+k}$, ${\mathcal X}$ a family over a smooth base
$Y^k$. We assume that the generic rank of $M$, $N$ is $g \le p$.  Let
$P(M)$ denote $\Projan {\mathcal R}(M)$, $\pi_M$
 the projection to ${\mathcal X}$.

We will be interested in computing, as we move from the special point $0$ to a generic point, the change in the multiplicity of
the pair $(M,N),$ denoted $\Delta(e(M,N))$. We will assume that the integral closures of $M$ and $N$ agree off a set $C$ of dimension
$k$ which is finite over $Y$, and assume we are working on a
sufficiently small neighborhood of the origin, so that every component
of $C$ contains the origin in its closure. Then $e(M,N, y)$ is the
sum of the multiplicities of the pair at all points in the fiber of
$C$ over $y$, and $\Delta(e(M,N))$ is the change in this number from
$0$ to a generic value of $y.$ If we have a set $S$ which is finite
over $Y$, then we can project $S$ to $Y$, and the degree of the
branched cover at $0$ is $mult_{y} S.$ (Of course, this is just the
number of points in the fiber of $S$ over our generic $y.$)

Let $C(M)$ denote the locus of points where $M$ is not free, {\it i.e.}, the
points where the rank of $M$ is less
than $g$, $C(\Projan {\mathcal R}(M))$
its inverse image under $\pi_M$.

We can now state the Multiplicity Polar Theorem. The proof in the ideal case appears in \cite{Gaff1}; the general proof appears in \cite{Gaff}.

\begin{theorem}(Multiplicity Polar Theorem) Suppose in the above
setup we have that $\overline{M} = \overline{N}$ off a set $C$ of
dimension $k$ which is finite over $Y$. Suppose further that
$C(Projan\mathcal{R}(M))(0) = C(Projan\mathcal{R}(M(0)))$ except
possibly at the points which project to $0 \in \mathcal{X}(0).$
Then, for y a generic point of $Y$, $$\Delta(e(M,N)) =
mult_{y}\Gamma_{d}(M) - mult_{y}\Gamma_{d}(N)$$
where ${\mathcal X}(0)$ is the fiber over $0$ of the family ${\mathcal X}^{d+k}$, $C(Projan\mathcal{R}(M))(0)$ is the fiber of $C(Projan\mathcal{R}(M))$ over $0$ and $M(0)$ is the restriction of the module $M$ to  ${\mathcal X}(0)$. \end{theorem}
\vspace{2cm}
\section{The Lipschitz saturation of an ideal and the definition of the $iL$ conditions}

The construction of the integral closure of an ideal is an example of a general approach to constructing closure operations on sheaves of ideals and modules given a closure operation on a sheaf of rings. Here is the idea. Denote the closure operation on the ring $R$ by $C(R)$. Given a ring, R, blow-up $R$ by an ideal $I$. (If we have a module $M$ which is a submodule of a free module $F$, form the blow-up $B_{\rho(\mathcal{M})}(\Projan \mathcal{R}(F))$, as in the last section.) Use the projection map of the blow-up to the base to pullback $I$ to the blow-up. Now apply the closure operation to the structure sheaf of the blow-up, and look at the sheaf of ideals generated by the pull back of $I$. The elements of the structure sheaf on the base which pull back to elements of the ideal sheaf are the elements of $C(I)$.

 Two examples of this are given by the normalization of a ring and the semi-normalization of a ring. (In the normalization, all of the bounded meromorphic functions become regular, while in the semi-normalization only those which are continuous become regular. Cf \cite{GV} for details on this construction.)  Consider $B_I(X)$, the blow-up of $X$ by $I$. If we pass to the normalization of the blow-up, then $h$ is in $\bar I$ iff and only if the pull back of $h$ to the normalization is in the ideal generated by the pullback of $I$ \cite{LT}. If we pass to the semi-normalization of the blow-up, then $h$ is in the weak sub-integral closure of $I$ denoted ${}^*I$, iff the pullback of $h$ to the semi-normalization is in the ideal generated by the pullback of $I$. (For a proof of this and more details on the weak subintegral closure cf. \cite {GV}).
 
 There is another way to look at the closure operation defined above; in the case of the integral closure of an ideal, we are looking at an open cover of the co-support of an ideal sheaf, and choosing locally bounded meromorphic functions on each open set, and seeing if we can write a regular function locally in terms of generators of the ideal using our locally bounded meromorphic functions as coefficients. This suggests, that in the Lipschitz case, we use locally bounded meromorphic functions which satisfy a Lipschitz condition. The closure operation on rings that this indicates is the Lipschitz saturation of a space, as developed by Pham-Teissier (\cite {PT}). 
 
In the approach of Pham-Teissier, let $A$ be a commutative local ring over ${\Bbb C}$, and $\bar A$ its normalization. (We can assume $A$ is the local ring of an analytic space $X$ at the origin in ${\Bbb C}^n$.) Let $I$ be the kernel of the inclusion
$$\bar A\otimes_{{\Bbb C}}\bar A\to \bar A\otimes_{A}\bar A.$$

In this construction, the tensor product is the analytic tensor product which has the right universal property for the category of analytic algebras, and which gives the analytic algebra for the analytic fiber product. 

Pham and Teissier then defined the Lipschitz saturation of $A$, denoted $\tilde A$, to consist of all elements $h\in \bar A$ such that $h\otimes 1-1\otimes h \in \bar A\otimes_{{\Bbb C}}\bar A$ is in the integral closure of $I$.  (For related results see \cite {L}.)

The connection between this notion and that of Lipschitz functions is as follows. If we pick generators $(z_1,\dots,z_n)$ of the maximal ideal of the local ring $A$, then $z_i\otimes 1-1\otimes z_i \in \bar A\otimes_{{\Bbb C}}\bar A$ give a set of generators of $I$. Choosing $z_i$ so that they are the restriction of coordinates on the ambient space,  the integral closure condition is equivalent to  
$$|h(z_1,\dots,z_n)-h(z'_1,\dots,z'_n)|\le C sup_i|z_i-z'_i|$$
holding on some neighborhood $U$, of $(0,0)$ on $X\times X$. This last inequality is what is meant by the meromorphic function $h$ being Lipschitz at the origin on $X$. (Note that the integral closure condition is equivalent to the inequality holding on a neighborhood $U$ for some $C$  for any set of generators of the maximal ideal of the local ring $A$. The constant $C$ and the neighborhood $U$ will depend on the choice.)

If $X,x$ is normal, then passing to the Lipschitz saturation doesn't add any functions. Denote the saturation of the blow-up by $SB_I(X)$, and the map to $X$ by $\pi_S$. Then we make the definition:

\begin{definition} let $I$ be an ideal in $\cO_{X,x}$, then the {\bf Lipschitz saturation} of the ideal $I$, denoted $I_S$, is the ideal
$I_{S}=\{h\in \cO_{X,x}| \pi^*_S(h)\in \pi^*_S(I)\}$.
\end{definition}

Since the normalization of a local ring $A$ contains the seminormalization of $A$, and the seminomalization contains the Lipschitz saturation of $A$, it follows that 
$\bar I\supset {}^*I\supset I_S\supset I$. In particular, if $I$ is integrally closed, all three sets are the same.

Here is a viewpoint on the Lipschitz saturation of an ideal $I$, which will be useful later. Given an ideal, $I$, and an element $h$ that we want to check for inclusion in $I_S$, we can consider $(B_I(X), \pi)$, $\pi^*(I)$ and $h\circ \pi$. Since $\pi^*(I)$ is locally principal, working at a point $z$ on the exceptional divisor $E$, we have a local generator $f\circ \pi$ of  $\pi^*(I)$. Consider the quotient $(h/f)\circ \pi$. Then $h\in I_S$ if and only if at the generic point of any component of $E$, $(h/f)\circ \pi$ is Lipshitz with respect to a system of local coordinates. If this holds we say $h\circ \pi\in (\pi^*(I))_S$.

We can also work on the normalized blow-up, $(NB_I(X), \pi_N)$. Then we say $h\circ \pi_N\in (\pi_N^*(I))_S$ if $(h/f)\circ \pi_N$  satisfies a Lipschitz condition at the generic point of each component of the exceptional divisor of $(NB_I(X), \pi_N)$ with respect to the pullback to $(NB_I(X), \pi_N)$ of a system of local coordinates on $B_I(X)$ at the corresponding points of $B_I(X)$. As usual, the inequalities at the level of $NB_I(X)$ can be pushed down and are equivalent to inequalities on a suitable collection of open sets on $X$. 

This definition can be given an equivalent statement using the theory of integral closure of modules. Since Lipschitz conditions depend on controlling functions at two different points as the points come together, we should look for a sheaf defined on $X\times X$.  We describe a way of moving from a sheaf of ideals on $X$ to a sheaf on $X\times X$.  Let $h\in \cO_{X,x}$; define $h_D$  in $ \cO^2_{X\times X,{x,x}}$, as $(h\circ \pi_1,h\circ\pi_2)$, $\pi_i$ the projection to the i-th factor of the product. Let $I$ be an ideal in $\cO_{X,x}$;  then $I_D$ is the submodule of $\cO^2_{X\times X,{x,x}}$ generated by the $h_D$ where $h$ is an element of $I$.

If $I$ is an ideal sheaf on a space $X$ then intuitively, $h\in \bar I$ if $h$ tends to zero as fast as the elements of $I$ do as you approach a zero of $I$. If $h_D$ is in ${\overline {I_D}}$ then the element defined by $(1,-1)\cdot (h\circ \pi_1,h\circ\pi_2)=h\circ \pi_1-h\circ\pi_2$ should be in the integral closure of the ideal generated by applying $(1,-1)$ to the generators of $I_D$, namely the ideal generated by $g\circ \pi_1-g\circ\pi_2$, $g$ any element of $I$. This implies the difference of $h$ at two points goes to zero as fast as the difference of elements of $I$ at the two points go to zero as the points approach each other. It is reasonable that elements in $I_S$ should have this property. In fact we have:

 \begin {theorem} Suppose $(X,x)$ is a complex analytic set germ, $I\subset\cO_{X,x}$. Then $h\in I_{S}$ if and only if 
$h_D\in{\overline I_D}$.
\end{theorem}
\begin {proof} This is theorem 2.3 of \cite {GL1}, and is proved there under the additional assumption that $h\in\bar I$. However, as we have noted if $h\in I_{S}$, then $h\in\bar I$. If $h_D\in{\overline I_D}$, it follows that $(1,0)\cdot h_D$ is in the integral closure of $\pi_1^*(I)$ on $X\times X$, which clearly implies $h\in\bar I$.
\end{proof}

Here is an example showing the difference between the integral closure of the Jacobian ideal and its saturation. Consider $f(x,y)=x^2+y^p$, $p>3$ odd. Denote the plane curve defined by $f$ by $X$. Then $X$ has a normalization given by $\phi =(t^p,t^2)$. The elements in the integral closure of the Jacobian ideal are just those ring elements $h$ such that $h\circ \phi\in \phi^*(J(f))=(t^p)$. Now $y^q\circ \phi=t^{2q}$, so $y^q\in{\overline {J(f)}}$ for $q>p/2$. Denote a matrix of generators for $J(f)_D$ by $[J(f)_D]$. Consider the curve mapping into $X\times X$ given by  $\Phi(t)=(t^p,t^2,t^p,ct^2)$, where $c$ is a $p$-th root of unity different from $1$. Now consider the ideal generated by the entries of the vector
$$<1,-1> [J(f)_D]\circ \Phi(t).$$
This ideal is generated by $(y^{p-1}-y'^{p-1}, (x,x',y^{p-1},y'^{p-1})(y-y'))\circ \Phi(t)=(t^{p+2})$. Meanwhile the order in $t$ of $<1,-1>(y^q, y'^q)\circ \Phi(t)=2q$. If $p<2q<p+2$ ie. $q=(p+1)/2$, then $(y^q, y'^q)$ cannot be in ${\overline {J(f)_D}}$, hence $y^q\notin J(f)_S$ but $y^q$ is in ${\overline {J(f)}}$.

Because we have re-cast the Lipschitz saturation of an ideal in integral closure terms, the invariants associated with integral closure become available to describe/control the Lipschitz saturation of an ideal. Notice first that the multiplicity of an ideal doesn't help, because the multiplicity of $I_S$ is same as the multiplicity of $I$ since they have the same integral closure. 

Even if $X$ is an isolated hypersurface singularity, $J(f)_D$ will not have finite colength, even in the plane curve case. The co-support will be $X\times {0}\cup {0}\times X\cup \Delta X$ in $X\times X$. However the multiplicity of the pair offers a way around this. The module $J(f)_D$ has a simple description, as we will see, off the origin in each of these three sets, and any integral closure condition we wish to use is easily checked because of this structure. This suggests looking for the largest module whose integral closure agrees with $J(f)_D$ off the origin, and using the multiplicity of the pair as our invariant. In the notation of \cite{G-5}, this module is denoted $H_{2n-3}(J(f)_D)$. This is the integral hull of $J(f)_D$ of codimension  $2n-3$, which means the integral closure of $J(f)_D$ and $H_{2n-3}(J(f)_D)$ agree off a set of codimension $2n-2$, ie. off $(0,0)$ in $X^{n-1}\times X^{n-1}$ . The next lemma identifies $H_{2n-3}(J(f)_D)$.

\begin{lem} Suppose $X^{n-1}$ is an isolated hypersurface singularity, defined by $f$. Then $H_{2n-3}(J(f)_D)={\overline {J(f)}}_D$.
\end{lem} 
\begin{proof}
We'll show that the integral closure of $J(f)_D$ and ${\overline {J(f)}}_D$ agree off the origin in $X\times X$.

Suppose $p=(x,x')\notin X\times {0}\cup {0}\times X\cup \Delta X$. Then for some $i,j,k$, $f_j(x)(z_i-z'_i)$ and $f_k(x')(z_i-z'_i)$ are not zero at $p$. This implies that both modules have rank $2$ at $p$, hence are equal. 

Suppose $p\in \Delta_X$, $p\neq (0,0)$; then for some $i$, $f_i(x)\neq 0 $. This implies $I_{\Delta}\oplus I_{\Delta}$ is in both modules. Further by adding elements of the form $(0, f_i(z)-  f_i(z'))$ which are in $I_{\Delta}\oplus I_{\Delta}$ to $( f_i(z),  f_i(z'))$, we see both modules contain $(1,1)$. Since both modules are contained in the module generated by $(1,1)$ and $I_{\Delta}\oplus I_{\Delta}$, and this module is integrally closed, the result is checked on $\Delta_X-(0,0)$.

Suppose $p=(x,0)$, $\neq 0$. Since $x\neq 0$, $J(f)_D$ contains $(1,0)$ and $(0,J(f))$.Thus $${\overline {J(f)_D}}={\cO}_{X,x}\oplus \overline {J(f)}={\overline {J(f)}}_D.$$

\end{proof}

The lemma suggests that it is interesting to consider the multiplicity of the pair $J(f)_D,{\overline {J(f)}}_D$, and we will use this invariant in the last section in the study of hyperplane sections of $X$.
For now we remark as a corollary of the proof of the lemma, we have for any $I$ an ideal of finite colength in any $\cO_X^d$,  that $H_{2d-1}(I)=(\overline I)_D$.
As a  corollary we have:

\begin{cor} Suppose $I\subset J\subset \overline I$ are ideals in $\cO_{X,x}$, with $X,x$ equidimensional, then $e(I_D, \overline I_D)=e(J_D, \overline I_D)$ if and only if 
${\overline I_D}={\overline J_D}$. 
\end{cor}
\begin{proof} From the additivity of multiplicity of pairs \cite{KT} it follows that $e(I_D, J_D)=0$ which is equivalent to their integral closures being the same.
\end{proof}

\begin {cor}  Suppose $I\subset J\subset \overline I$ are ideals in $\cO_{X,x}$, with $X,x$ equidimensional, then $e(I_D, \overline I_D)=e(J_D, \overline I_D)$ if and only if $I_S=J_S$.
\end {cor}
\begin{proof} This follows from the connection between the Lipschitz saturation of an ideal and integral closure.
\end{proof}

Now we add the necessary structure to deal with families of spaces.

Just as Pham-Teissier extended their original definition to a family of spaces, we can do the same.
Suppose $X^{d+k},0$ is an analytic space containing a smooth subset $Y^k,0$, and $(X^{d+k}, p)$ is a family of spaces over $Y$, $X$, $Y$ embedded in 
$\mathbb{C}^{n+k},0$, so that $p$ is the projection on the last $k$ factors of $\mathbb{C}^{n+k},0$, where $Y^k=0\times \mathbb{C}^{k}$.

Then, in the definition of the Lipschitz saturation rel $Y$ of the local ring of $X^{d+k},0$, we use a set of local coordinates on the ambient space which restrict to generators of the maximal ideals of the fibers  of $X$ over $Y$. This amounts to looking at the fiber product of the normalization of $X$ with itself over $Y$, and asking that locally $h\circ p_1-h\circ p_2$ is in the integral closure of the double of the ideal generated by these coordinates. 

Given an ideal sheaf $I$ on $X^{d+k},0$, using the relative saturation, we can define the Lipschitz saturation of $I$ relative to $Y$. When we are working in the context of a family of spaces we will also use $I_S$ to denote this saturation. In a similar way, we can develop an equivalent integral closure condition using modules as before, just working on $X\times_Y X$ instead of $X\times X$.

In practice we will be working with ideal sheaves on a family of spaces, where the ideals vanish on $Y$, and our local coordinates at points of $B_I(X^{n+k})$ consist of the pullbacks of a set of generators of $m_Y$ and local coordinates on the projective space(s) in the blow-up.  

It is not difficult to check that Theorem 2.3 of \cite{GL1} continues to hold in this new context.

Having constructed the necessary infinitesimal objects we now develop our condition.
\vskip .1in
\noindent {\it Setup} Let $X^{n+k},0 \subset \mathbb{C}^{n+1+k},0$ be a hypersurface, containing a smooth subset $Y$ embedded in $\mathbb{C}^{n+1+k}$ as $0\times \mathbb{C}^{k}$, with $p_Y$ the projection to $Y$. Assume $Y=S(X)$, the singular set of $X$. Suppose $F$ is the defining equation of $X$, $(z,y)$ coordinates on $\mathbb{C}^{n+1+k}$. Denote by $f_y(z)=F(z,y)$ the family of functions of defined  by $F$, and by $X_y$, $f^{-1}_y(0)$. Assume $f_y$ has an isolated singularity at the origin. Let $m_Y$ denote the ideal defining $Y$, and $J(F)_Y$, the ideal generated by the partial derivatives with respect to the $y$ coordinates, $J_z(F)$, those with respect to the $z$ coordinates. 

\begin{definition} The pair $(X,Y)$ satisfy the $iL_{m_Y}$ condition at the origin if either of the two equivalent conditions hold:

1) $J(F)_Y\subset (m_YJ_z(F))_S$

2)  $(J(F)_Y)_D\subset {\overline{(m_YJ_z(F))_D}}$.
\end{definition}

An analogous condition for $iL_{m_Y}$ is  $J(F)_Y\subset {\overline {m_YJ_z(F)}}$. This is the equivalent to the Verdier's condition W or the Whitney conditions.

Next we give the definition of $iL_A$.

\begin{definition} The pair $(X,Y)$ satisfy the  $iL_A$, at the origin if either of the two equivalent conditions hold:

1) $J(F)_Y\subset (J_z(F))_S$

2)  $(J(F)_Y)_D\subset {\overline{J_z(F))_D}}$.
\end{definition}

The analogous condition is  $J(F)_Y\subset {\overline{J_z(F)}}$. If one works on the ambient space, then this is equivalent to the A$_F$ condition. Working on $X$, it is equivalent to asking that the $X$ has no vertical tangent plane at the origin, so this is weaker than Whitney A. However, suppose $l$ is a linear form on the ambient space. Let $J(F)_l$ denote the ideal generated by applying tangent vectors in the kernel of $l$ to $F$. So $J_z(F)=J(F)_y$ in the case dim $ Y=1$. Working in the one dimensional parameter case, if there exist a pencil of forms $l_s$  including $y$ such that 
$J(F)\subset {\overline{J(F)_{l_s}}}$ then not only does Whitney A hold but the total space has no relative polar curve. This follows because if the dimension of the fiber of the limiting tangent hyperplanes over the origin is not maximal then the fiber over the origin must be in the closure of the fiber over the parameter space with $y\ne 0$, and all of these hyperplanes contain $Y$. Because the dimension of the fiber over the origin is less than maximal this also implies the polar curve is empty.  The condition with the pencil of forms ensures that no hyperplane defined by an element of the pencil can be a limiting tangent hyperplane, hence the pencil of hyperplanes has no intersection with the fiber over zero, which must therefore have less than maximal dimension.

Since there are different ways in which the total space $X^{n+k}$ can be made into a family of spaces, it is natural to ask if the  conditions we have defined  depend on the projection to $Y$ which defines the family.  We now show that the condition $iL_{m_Y}$ does not depend  on the projection to $Y$.

\begin{Prop} In the above set-up the following conditions are equivalent. 

1) $(J(F)_Y)_D\subset {\overline{(m_YJ_z(F))_D}}$.

2) $(J(F)_Y)_D\subset {\overline{(m_YJ(F))_D}}$.
\end {Prop}

The analogous result for W is quite easy. The Lipschitz case is more technical. We first show:

\begin{Lemme} In the above setup if   $(J(F)_Y)_D\subset {\overline{(m_YJ(F))_D}}$, then  $J(F)_Y\subset{\overline{ m_YJ(F)}}$, hence condition $W$ holds for the pair $(X-Y, Y)$ at the origin (and hence on some Z-open subset of $Y$ containing the origin.)
\end{Lemme}

\begin{proof}  We use the curve criterion. We can choose a curve $\Phi=(\phi_1, \phi_2)$, where $\phi_1$ maps $\Bbb C,0$ to $0$, and $\phi_2$ is arbitrary. Then the curve criterion for this curve becomes $\phi^*_2(J(F)_Y)\subset  \phi^*_2(m_YJ(F))$. Here an easy argument using Nakayama's lemma implies that $\phi^*_2(J(F)_Y)\subset  \phi^*_2(m_YJ_z(F))$, which implies the $W$ condition.

\end{proof}

Now we prove our proposition.

\begin{proof} We use the curve criterion again. Let $\Phi=(\phi_1, \phi_2)$. It is enough to prove it in the case where $Y$ is one dimensional, since the notation is the only part of the proof which is harder in general. It is also clear that 1) implies 2), so we assume 2). By the given we have:

$$  (\frac{\partial F}{\partial y})_D\circ\Phi=\sum g_{i,j}(t)(z_i\frac{\partial F}{\partial {z_j}})_D\circ \Phi+\sum  g_{i,j,k}(t)(z_k\circ \phi_1-z_k\circ\phi_2)(0, z_i\frac{\partial F}{\partial {z_j}})\circ\phi_2$$

$$+\sum h_i(t)(z_i  \frac{\partial F}{\partial y})_D \circ\Phi.$$

We now work mod $m_1\Phi^*(m_YJ(F)_D)$ and we call the left side of the above equation $*$. Subtract $\sum h_i(t) z_i\circ \phi_1 *$ from both sides of the above equation. This sum is in $m_1\Phi^*(m_YJ(F)_D)$, so we get:

$$( \frac{\partial F}{\partial y})_D\circ\Phi=\sum g_{i,j}(t)(z_i\frac{\partial F}{\partial {z_j}})_D\circ \Phi+\sum  g_{i,j,k}(t)(z_k\circ \phi_1-z_k\circ\phi_2)(0, z_i\frac{\partial F}{\partial {z_j}})\circ\phi_2$$

$$+\sum h_i(t)(z_i \circ \phi_2-z_i\circ \phi_1)(0, \frac{\partial F}{\partial y} \circ\phi_2).$$

Now we use the lemma to write $\frac{\partial F}{\partial y} \circ\phi_2$ as an element of $\phi_2^*(m_YJ_z(F))$. Making the substitution into the line above shows that the terms there are $0$ mod \hskip 2pt$m_1\Phi^*(m_YJ(F)_D)$, hence we have $ \frac{\partial F}{\partial y}\circ\Phi$ is an element of $(m_YJ_z(F))_D$ mod\hskip 2pt $m_1\Phi^*(m_YJ(F)_D)$. Hence by Nakayama's lemma, $\Phi^*m_YJ_z(F))_D=
\Phi^*m_YJ(F))_D$
and the proposition follows.
 \end{proof}

While a similar result for $iL_A$ doesn't make sense, if we ask that $(J(F)_Y)_D$ is strictly dependent on $J_z(F))_D$ then an analogous result holds.  (Recall that an element $h\in \cO^p_{X,x}$ is strictly dependent on $M\subset \cO^p_{X,x}$, if for each curve $\phi$ $h\circ \phi\in m_1\phi^*(M)$. The set of elements strictly dependent on $M$ are denoted $M^+$.) 
\vskip .2in

We give a geometric interpretation of these conditions at the level of the family $X^{n+k}$. We make some preliminary constructions to do this. Denote the coordinates on  $\mathbb P^n$ by $T_i$, $1\le i\le n+1$, let $V_i$ be the subset of $\mathbb P^n$ defined by $T_i\ne 0$, and let $U_i$ denote $B_{J_z(F)}(X^{n+k})\cap (X\times V_i)$. At each point of $U_i$, $\pd F{z_i}\circ \pi$ is a local generator of the principal ideal sheaf $\pi^*(J_z(F))$. The condition that  $\pd F{y_j}$ be in the Lipschitz saturation of $J_z(F))$ means that at each point of $U_i$, 
${{\pd F{y_j}}\over {\pd F{z_i}}}\circ \pi$ is Lipschitz rel $Y$ with respect to the local coordinates, which are $z_k\circ \pi$, $1\le k\le n+1$, and $T_j/T_i$, $1\le j\le n+1$, $j\ne i$. Since ${{\pd F{z_j}}\over {\pd F{z_i}}}\circ \pi={{T_j}\over{T_i}}$, this implies that
${{\pd F{y_j}}\over {\pd F{z_i}}}$ is Lipschitz with respect to $z_k$, $1\le k\le n+1$, and ${{\pd F{z_j}}\over {\pd F{z_i}}}$, $1\le j\le n+1$, $j\ne i$ on $\pi(U_i)$.

This implies the existence of $k$ vectorfields tangent to $X$ defined on each $\pi(U_i)$ of the form 
$$\vec v_{j,i}={\pd {}{y_j}} -{{\pd F{y_j}}\over {\pd F{z_i}}}{\pd {}{z_i}},$$
each vectorfield Lipschitz relative to $Y$,with respect to $z_k$, $1\le k\le n+1$, and ${{\pd F{z_j}}\over {\pd F{z_i}}}$, $1\le j\le n+1$, $j\ne i$. Since every element of $J_z(X)$ is in the Lipschitz saturation of $J_z(X)$ it is not true apriori that these vectorfields are extensions of the constant fields on $Y$. However, if we assume the A$_F$ condition holds for $(X-Y,Y)$, then the quotients ${{\pd F{y_j}}\over {\pd F{z_i}}}\circ \pi$ will vanish on the exceptional divisor, and the $\vec v_{j,i}$ will be extensions of the constant fields on $Y$. 

There is another useful interpretation which we can make. Recall
the following definition of distance between two linear subspaces A, B at the origin in ${\mathbb C}^{N} $, then
			 $${ \text{\rm{dist }}}(A,B) = \mathop {\sup}\limits_{\begin{matrix} u\in {B^{\perp}}-\{0\}\\ {v\in {A-\{0\}}}\end{matrix}}
{{|(u,v)|} \over {\left\| u \right\|\left\| v \right\|}}.$$

If $p$, $p'$ are smooth points in the same fiber $y$ over $Y$ in $\pi(U_i)$, we claim that the distance between the tangent spaces to $X$ at $p$ and $p'$ is commensurate with the maximum of the distance between the tangent spaces to $X_y$ at $p$ and $p'$ and the distance between the points.

We first relate the distance defined above to a notion of distance closer to our Lipschitz condition.

Suppose ${\bf a}=(a_0,\dots, a_n)$, ${\bf b}=(b_0,\dots, b_n)$ define hyperplanes $A$ and $B$ in $\Bbb C^{n+1}$.
We will use the supnorm  on $\Bbb C^{n+1}$; suppose $||{\bf a}||=a_i$, and  $||{\bf b}||=b_i$, same index for both, for simplicity take $i=0$.

We can then also measure the distance between $A$ and $B$ by using the $ \mathop {\sup}\limits_{i, i\le i\le n} ||a_i/a_0-b_i/b_0||$. The $a_i/a_0$ are just the coordinates of the hyperplane $A$ regarded as  a point of  ${\hat{ \Bbb P}}^n$. We compare this notion of distance with the usual one.

\begin{lem} Suppose ${\bf a}=(a_0,\dots, a_n)$, ${\bf b}=(b_0,\dots, b_n)$ define hyperplanes $A$ and $B$ in $\Bbb C^{n+1}$, $||{\bf a}||=a_0$, and  $||{\bf b}||=b_0$. Then
$${ \text{\rm{dist }}}(A,B) = \mathop {\sup}\limits_{i, 1\le i\le n} ||a_i/a_0-b_i/b_0||.$$
\end{lem}
\begin{proof} A basis for the vectors in $A$ are given by $a_0e_i-a_ie_0$ where $e_k$ is the $k$-th standard basis vector in $\Bbb C^{n+1}$. Since we are using the supnorm, the terms $${{|(u,v)|} \over {\left\| u \right\|\left\| v\right\|}},$$ become 
$${{|(a_0e_i-a_ie_0,\bar {\bf b})|} \over {\left\| a_0 \right\|\left\| b_0 \right\|}}=||a_i/a_0-b_i/b_0||$$

\end{proof}

Now we return to our geometric interpretation. Since the $\pd F{y_i}$ are in the integral closure of $J_z(F)$, we may work in a system of neighborhoods $U_i$ on $X$ where we may assume for each $p\in U_i$ the values of the elements of $J(F)$ are bounded in norm by $|\pd F{z_i}(p)|$. Then, applying the above lemma, we see that the distance between tangent planes to $X$ at points $p_1$, $p_2$ in the same $U_i$ is the sup over 
$$\{||{{\pd  F{y_k}(p_1)}\over{\pd F{z_i}(p_1)}}-{{\pd  F{y_k}(p_2)}\over{\pd F{z_i}(p_2)}}||, ||{{\pd  F{z_j}(p_1)}\over{\pd F{z_i}(p_1)}}-{{\pd  F{z_j}(p_2)}\over{\pd F{z_i}(p_2)}}||\}.$$

Then condition $iL_A$ implies that this is the same as the sup over 
$$\{||{{\pd  F{z_j}(p_1)}\over{\pd F{z_i}(p_1)}}-{{\pd  F{z_j}(p_2)}\over{\pd F{z_i}(p_2)}}||\ , ||p_1-p_2||\}$$
which is the same as the maximum of the distance between the tangent spaces to $X_y$ at $p_1$ and $p_2$ and the distance between the points, $p_1$ and $p_2$.

We can say something similar for the $iL_W$ condition. First, since $iL_W$ implies $iL_A$, the same interpretation applies to the $iL_W$ condition.  But more is true, and we develop some material related to the Lipschitz saturation of the product of two ideals to explain it. 

\begin{lem} (Product lemma) Given $h$,$g$ in $\mathcal{O}_{X,x}$, $p_1$,$p_2\in X$, then 
$$\|(hg)(p_1)-(hg)(p_2)\|\le  \|h(p_1)\|\|g(p_1)-g(p_2)\|+$$
$$\hskip 2in \|g(p_2)\|\|h(p_1)-h(p_2)\|.$$
\end{lem}
\begin{proof} We have 
$$\|(hg)(p_1)-(hg)(p_2)\|=\|(hg)(p_1)-h(p_1)g(p_2)+ h(p_1)g(p_2)-(hg)(p_2)\|$$ 

$$=\|h(p_1)(g(p_1)-g(p_2))+ g(p_2)(h(p_1)-h(p_2))\|$$
$$\le \|h(p_1)\|\|g(p_1)-g(p_2)\|+\|g(p_2)\|\|h(p_1)-h(p_2)\|$$

\end{proof}
Note that we can always choose one of the terms, say $\|g(p_i)\|$, to be the minimum of the $\|g(p_i)\|$. (You cannot, in general, minimize both $h$ and $g$ terms.)

We apply this lemma to the condition for $h\in \mathcal{O}_{X,x}$ to be in the Lipschitz saturation of $IJ$,  $I$,$J$ two ideals of $\mathcal{O}_{X,x}$.

Suppose $I=(f_1,\dots,f_p)$, $J=(g_1,\dots,g_q)$. Work on the Zariski open subset $U_{m,n}$ of $(B_{IJ}(X),\pi)$
in which $(f_mg_n)\circ \pi$ is a local generator of $\pi^*(IJ)$. Local coordinates are given by the pullback of coordinates at $x$, and by $T_{i,j}$ where $(i,j)\ne(m,n)$, $1\le i\le p$,$1\le j\le q$, and where 
$$T_{i,j}={{(f_ig_j)\circ\pi}\over{{f_mg_n}\circ\pi}}$$

Note that $$T_{m,j}={{(f_mg_j)\circ\pi}\over{{f_mg_n}\circ\pi}}={{g_j\circ\pi}\over{{g_n}\circ\pi}}$$
while
$$T_{i,n}={{(f_ig_n)\circ\pi}\over{{f_ig_n}\circ\pi}}={{f_i\circ\pi}\over{{f_m}\circ\pi}}.$$

The next lemma shows that among all the $T_{i,j}$, on $U_{m,n}$ we need only consider the $T_{m, j}$ and $T_{i,n}$ to define the Lipschitz saturation of $IJ$.
As usual, $\pi_N$ denotes the normalization map, while $p_1$ and $p_2$ are projection maps from the product of the normalization of $B_{IJ}(X)$ with itself.

\begin{lem} Let $U_{m,n}$ be as above, then the ideal generated by $$\{ T_{j,n}\circ \pi_N\circ p_1- T_{j,n}\circ \pi_N\circ p_2,  T_{m,i}\circ \pi_N\circ p_1- T_{m,i}\circ \pi_N\circ p_2\}, 1\le j\le p, j\ne m, 1\le i\le q, i\ne n$$ is a reduction of the ideal generated by $$\{T_{j,i}\circ \pi_N\circ p_1- T_{j,i}\circ \pi_N\circ p_2\}$$ at points of $\pi^{-1}_N(U_{m,n})\times \pi^{-1}_N(U_{m,n})$.
\end{lem}
\begin{proof} By the product lemma we have 

$$\|{{f_ig_j}\over{{f_mg_n}}}\circ\pi\circ\pi_N\circ p_1(z'_1,z'_2)-{{f_ig_j}\over{{f_mg_n}}}\circ\pi\circ\pi_N\circ p_2(z'_1,z'_2)\|$$
$$ \le \|{{f_i\circ\pi}\over{{f_m}\circ\pi}}\circ \pi_N\circ p_1(z'_1,z'_2)\|\|{{g_j\circ\pi}\over{{g_n}\circ\pi}}\circ\pi_N\circ p_1(z'_1,z'_2)-{{g_j\circ\pi}\over{{g_n}\circ\pi}}\circ \pi_N\circ p_2(z'_1,z'_2)\|$$
$$+ \|{{g_j\circ\pi}\over{{g_n}\circ\pi}}\circ \pi_N\circ p_1(z'_1,z'_2)\|\|{{f_i\circ\pi}\over{{f_m}\circ\pi}}\circ\pi_N\circ p_1(z'_1,z'_2)-{{f_i\circ\pi}\over{{f_n}\circ\pi}}\circ \pi_N\circ p_2(z'_1,z'_2)\|$$
Now we can bound the terms $ \|{{f_i\circ\pi}\over{{f_m}\circ\pi}}\circ \pi_N\circ p_1(z'_1,z'_2)\|$ and $\|{{g_j\circ\pi}\over{{g_n}\circ\pi}}\circ \pi_N\circ p_1(z'_1,z'_2)\|$ locally by constants because the ideal $IJ$ is principal on $U_{m,n}$. The result follows from this.

\end{proof}

We apply the above results to say something about the local vectorfields $\vec v_{i,j}$ defined above. Since ${\pd {F}{y_j}}\in (m_YJ_z(F)_S)$, we can usefully re-write $\vec v_{i,j}$ as 
$$\vec v_{i,j,k}={\pd {}{y_j}} -{{\pd F{y_j}}\over {z_k\pd F{z_i}}}z_k{\pd {}{z_i}}.$$
Denote the coefficient of ${\pd {}{z_i}}$ in $\vec v_{i,j,k}$ by $v_{i,j,k}$.

Then for pairs of points $(t,p_1),(t,p_2)$ in $\pi(U_{i,k})$ we have:
$$\|v_{i,j,k}(t,p_1)-v_{i,j,k}(t,p_2)\|\le \|{{\pd F{y_j}}\over {z_k\pd F{z_i}}}(t,p_1)\|\|z_k(p_1)-z_k(p_2)\|$$
$$+ \| z_k(p_2)\| {{\pd F{y_j}}\over {z_k\pd F{z_i}}}(t,p_1)-{{\pd F{y_j}}\over {z_k\pd F{z_i}}}(t,p_2)\|$$
Hence,
$$\|v_{i,j,k}(t,p_1)-v_{i,j,k}(t,p_2)\|\le C \|z_k(p_1)-z_k(p_2)\| $$
$$+ \| z_k(p_1)\| \mathop{sup}\{\|{{\pd F{z_j}}\over {\pd F{z_i}}}(t,p_1)-{{\pd F{z_j}}\over {\pd F{z_i}}}(t,p_2)\|, \|{{z_j}\over {z_k}}(p_1)-{{z_j}\over {z_k}}(p_2)\|\}.$$
Here we may assume that $ \| z_k(p_2)\|$ is the smaller of $ \| z_k(p_1)\|$, $ \| z_k(p_2)\|$.
So, if the local fields are not Lipschitz on $U_{i,k}$ with respect to the distance between points, then they are Lipschitz with respect the distance between planes or secant lines to the origin and in this case the Lipschitz constant goes to zero as one of the points goes to the origin.

\section{Genericity Theorem}

Although at present we can't give a complete proof that the $iL_{m_Y}$ condition is generic, we can do both conditions at once in some of the cases.  We first determine the different cases in which it is necessary to check the conditions. These cases are the different ways in which $J_z(F)_D$ can fail to have maximal rank.

\begin{Prop} The co-supports of $(m_YJ_z(F))_D$ or  $J_z(F)_D$ on $X\times_Y X$ consist of

1) $Y\times (0,0)$

2) $\Delta(X\times_Y X)$

3) $(0\times_Y X)\cup (X\times_Y 0)$

\end{Prop}
\begin{proof}

Suppose $(x,x')$ does not lie in one of the sets. Then,  since some $z_i\circ p_1$ and some $z_j\circ p_2$ are not zero at $(x,x')$, $(m_YJ_z(F))_D=J_z(F)_D$ locally. Then $J_z(F)_D$ contains terms of the form $(0,\frac{\partial F}{\partial {z_j}}\circ p_2)$, $(\frac{\partial F}{\partial {z_j}}\circ p_1,0)$, which implies that the rank of $(m_YJ_z(F))_D$ is 2 and $(x, x') $ are not in the cossupport.

\end{proof}

The reader may have noted that $Y\times (0,0)$ is a subset of both $\Delta(X\times_Y X)$ and $(0\times_Y X)\cup (X\times_Y 0)$. We will next show that generically both conditions hold at points of 
$\Delta(X\times_Y X)-Y\times (0,0)$, and of $(0\times_Y X)\cup (X\times_Y 0)-Y\times (0,0)$. Since we are working on a Z-open set of $Y$, and we are working with families of isolated singularities, we may assume that the only singular point of $X_y$ is at $(y,0)$, that $(X-Y,Y)$ satisfies $W$ at $(y,0)$. We will show that checking the conditions at points of the form $(y,0,x)$, $x\ne 0$ amounts to checking $W$ at 
$(y,0)$ for $(X-Y,Y)$ , while checking the conditions at points of $\Delta(X\times_Y X)$, $x\ne 0$ is trivial. Thus it will suffice to look at components of the appropriate exceptional divisor that surject onto $Y\times (0,0)$.

\begin{Prop} In the set-up of this section, $iL_A$ and $iL_{m_Y}$ hold at all points of $\Delta(X\times_Y X)-Y\times (0,0)$, and both conditions hold at all points of $(0\times_Y X)\cup (X\times_Y 0)-Y\times (0,0)$ such that $(X-Y,Y)$ satisfies $W$ at $(y,0)$.
\end{Prop}

\begin{proof} Work at  $(y,x,x)$, $x\ne 0$. Then since $x\ne 0$, $(m_YJ_z(F))_D=J_z(F)_D$ locally. Since $f_y$ is a submersion at $x$, and $J_z(F)_D$ contains elements of the form
$(0, (z_i\circ p_1-z_i\circ p_2) (\frac{\partial F}{\partial {z_j}}\circ p_2))$, $((z_i\circ p_1-z_i\circ p_2) (\frac{\partial F}{\partial {z_j}}\circ p_1),0)$, it follows that $J_z(F)_D$ contains $I_{\Delta}\cO^2_{X\times_Y X,(x,x)}$. By adding elements of the form $(0, \frac{\partial F}{\partial {y}}\circ p_1- \frac{\partial F}{\partial {y}}\circ p_2)$ to $( \frac{\partial F}{\partial {y}}\circ p_1, \frac{\partial F}{\partial {y}}\circ p_2)$   and elements of the form $(0, \frac{\partial F}{\partial {z_j}}\circ p_1-\frac{\partial F}{\partial {z_j}}\circ p_2)$ to $( \frac{\partial F}{\partial {z_j}}\circ p_1, \frac{\partial F}{\partial {z_j}}\circ p_2)$, this part of the proof is finished  since $ \frac{\partial F}{\partial {y}}$ is in the ideal $J_z(F)$ at $x$ since $f_y$ is a submersion.

Now work at $(x,0)$, $x\ne 0$. Since $f_y $ is a submersion at $x$, and $x\ne 0$ it follows that $(m_Y J_z(F))_D$ contains elements of the form $(1,0)$, so it suffices to show that $\frac{\partial F}{\partial {y}}$ is in the integral closure of $m_Y J_z(F))$ and this is equivalent to $W$. This ends the second part of the proof.

\end{proof}

\begin{theorem} In the set-up of this section, there exists a Zariski open subset of $U$ of $Y$ such that $iL_A$ holds for the pair $(X-Y, U\cap Y)$ along $Y$.
\end{theorem}

\begin{proof}
We will follow the lines of the proof of the Idealistic Bertini Theorem given in \cite{T1} p591-598. We prove that the $il_A$ condition is generic using the module criterion. We will work on the normalized blow-up of $X\times_Y X\times \Bbb P^1$ by the ideal sheaf induced from the submodule $J_z(F)_D$, denoting $NB_{(J_z(F))_D}( X\times_Y X\times \Bbb P^1$) by $N$. We need to check that on each component of the exceptional divisor that the pullback of the element induced from $ (\frac{\partial F}{\partial {y}})_D$ to the normalized blowup is in the pullback of $(J_z(F))_D$. Denote the projection  to $Y$ by $p$. By the previous lemmas we need only consider those components of the exceptional divisor which project to $Y$ under the map to $X\times_Y X$. Since we are working over a Zariski open subset of $Y$ we may assume that every such component maps surjectively onto $Y$. Since we are working on the normalization, we can work at a point $q$ of the exceptional divisor such that $E$ is smooth at  $q$, $N$ is smooth at $q$ and the projection to $Y$ is a submersion at $q$.  Thus, we can choose coordinates at $q$, $(y',u', x')$, such that $y'=y\circ p$, and $u'$ defines $E$ locally with reduced structure. The key point is that $ \frac{\partial u'}{\partial {y'}}=0$.

Let $\pi_i$ denote the composition of $\pi$, the projection from $N$ to $ X\times_Y X\times \Bbb P^1$ with the projection $p_i$ to the $i$-th factor of $X\times_Y X\times \Bbb P^1$, $i=1, 2$.

We have that $F\circ p_1+sF\circ p_2$ is identically zero on $ X\times_Y X\times \Bbb P^1$. Pull this back to $N$ by $\pi$ and take the partial derivative with respect to $y'$ at $q$. We get by the chain rule:

$$0=\frac{\partial F}{\partial {y}}\circ \pi_1+s\frac{\partial F}{\partial {y}}\circ \pi_2 + \sum\limits_{i=1}^{n}\frac{\partial F}{\partial {z_i}}\circ \pi_1 \frac{\partial {z_i\circ \pi_1}}{\partial {y'}}+s\frac{\partial F}{\partial {z_i}}\circ \pi_2 \frac{\partial {z_i\circ \pi_2}}{\partial {y'}}.$$

Notice that there is no term involving the derivative of $s$. This is because the coefficient of this partial by the product rule would be zero, since $F\circ \pi_i=0$.

Now we work to re-shape the above term to prove the theorem. Notice that since $z_i$ all vanish along $Y$, $z_i\circ \pi_j$ all vanish along $E$ at $q$.
We can assume the order of vanishing of $z_1\circ \pi_j$ is  minimal among $\{z_i\circ\pi_j\}$, and that the strict transforms of $z_1\circ \pi_j$ do not pass through $q$.

We have 

$$\frac{\partial F}{\partial {y}}\circ \pi_1+s\frac{\partial F}{\partial {y}}\circ \pi_2 =$$
$$-(\sum\limits_{i=1}^{n}(\frac{\partial F}{\partial {z_i}}\circ \pi_1) 
 (\frac{\partial {z_i\circ \pi_1}}{\partial {y'}})+s((\frac{\partial F}{\partial {z_i}}\circ \pi_2) (\frac{\partial {z_i\circ \pi_1}}{\partial {y'}})$$
$$-(\frac{\partial F}{\partial {z_i}}\circ \pi_2) \left[ \frac{\partial {z_i\circ \pi_1}}{\partial {y'}}-\frac{\partial {z_i\circ \pi_2}}{\partial {y'}}\right])).$$

We want to show that the terms on the right hand side in the above expression are in the ideal generated by the pullback of the ideal sheaf on $ X\times_Y X\times \Bbb P^1$ induced by $J_z(F))_D$. For this we use the curve criterion. We use a test curve to show that the order of vanishing of $\frac{\partial F}{\partial {y}}\circ \pi_1+s\frac{\partial F}{\partial {y}}\circ \pi_2$ along a component is same as the order of vanishing of the ideal $(J_z(F))_D$. This will imply that  $\frac{\partial F}{\partial {y}}\circ \pi_1+s\frac{\partial F}{\partial {y}}\circ \pi_2$ is in the ideal along the component.  We can choose a curve $\tilde{\Phi}$ such that $\tilde{\Phi}$ is the lift of a curve $\Phi=(\psi, \phi_1, \phi_2)$, $\Phi:{\Bbb C}:\to {\Bbb P}^1\times X\times_YX$. Further $\tilde\Phi(0)$ is a smooth point of the component and the ambient space, $\tilde\Phi$ transverse to the component so that $u'\circ \tilde\Phi=t$, where $t$ is a coordinate in the local ring of $\Bbb C$ at the origin. This implies that if an ideal is generated by  $u'^p$, that the pullback is generated by $t^p$. Since the pullback of the ideal $(J_z(F))_D$ is locally principal, we can choose $\tilde\Phi(0)$ so that $(J_z(F))_D$ is generated by a power of $u'$.

Then we have $$\tilde\Phi^*(\frac{\partial F}{\partial {y}}\circ \pi_1+s\frac{\partial F}{\partial {y}}\circ \pi_2)=$$

$$-(\sum\limits_{i=1}^{n}(\frac{\partial F}{\partial {z_i}}\circ \pi_1\circ \tilde\phi_1) 
 (\frac{\partial {z_i\circ \pi_1}}{\partial {y'}})\circ \tilde\phi_1+\psi_2/\psi_1((\frac{\partial F}{\partial {z_i}}\circ \pi_2)\circ \tilde\phi_2 (\frac{\partial {z_i\circ \pi_1}}{\partial {y'}})\circ \tilde\phi_1$$
$$-(\frac{\partial F}{\partial {z_i}}\circ \pi_2)\circ \tilde\phi_2 \left[ \frac{\partial {z_i\circ \pi_1}}{\partial {y'}}\circ \tilde\phi_1-
\frac{\partial {z_i\circ \pi_2}}{\partial {y'}}\tilde\phi_2\right])).$$

The right hand side will clearly be in the ideal ${\Phi}^*(J_z(F))_D)$, provided the pullback of $(\frac{\partial F}{\partial {z_i}}\circ \pi_2) ( \frac{\partial {z_i\circ \pi_1}}{\partial {y'}}-
\frac{\partial {z_i\circ \pi_2}}{\partial {y'}})$ is. However, by construction, since $y'$ and $u'$ are independent coordinates,  the order of 
$\frac{\partial {z_i\circ \pi_1}}{\partial {y'}}-\frac{\partial {z_i\circ \pi_2}}{\partial {y'}}$ in $u'$ will be the same as the order of $z_i\circ \pi_1 -z_i\circ \pi_2$. Hence the pullback of $(\frac{\partial F}{\partial {z_i}}\circ \pi_2) ( \frac{\partial {z_i\circ \pi_1}}{\partial {y'}}-
\frac{\partial {z_i\circ \pi_2}}{\partial {y'}})$ does vanish to the desired order in $t$, which finishes the proof.
\end{proof}

We describe an application of this result. Given $X$ an isolated hypersurface singularity we can consider the sections of $X$ by hyperplanes. It is natural to ask if there is a generic set of hyperplanes for which the associated family of hyperplane sections satisfies the $iL_A$ condition. We will show this is true after recalling the ideas necessary to make precise statements. (For more details on this material see \cite{G-1}.) We first need the notion of the Grassman modification of $X$, which we describe in the hyperplane case. Let $E_{n-1}$ denote the canonical bundle over ${\Bbb{P}}^{n-1}$, which we view as hyperplanes though the origin in ${\Bbb C}^n$. Denote the projection
of $E_{n-1}$ to ${\Bbb C}^n$ by $\beta_{n-1}$. If $X^{n-1}$ is a subset of ${\Bbb C}^n$, we call $\tilde X ={\beta_{n-1}}^{-1}(X)$, the $ G_{n-1}$ modification
of $X$. In this paper we will simply refer to the $G_{n-1}$ modification as the Grassman modification of $X^{n-1}$. Note that ${\Bbb{P}}^{n-1}$ is embedded in $E_{n-1}$ as the  zero section of $E_{n-1}$. This means that we 
can think of $0\times {\Bbb{P}}^{n-1}$ as a stratum of 
 $\tilde X$; note that the projection to $0\times{\Bbb{P}}^{n-1}$
makes $\tilde X$ a family of analytic sets with $0\times {\Bbb{P}}^{n-1}$ as the parameter space which we denote by $Y$. The members of 
this family are just
$\{P\cap X\}$ as $P$ varies through the points of ${\Bbb{P}}^{n-1}$. 

The set of hyperplanes which are limiting tangent planes to $X$ at the origin form a Zariski closed set. It is known that on the complement of this set, $(\widetilde X-Y,Y)$ are a pair of strata which satisfy the Whitney conditions. We can now apply Theorem 4.3 to this situation. 

\begin{theorem} Suppose $X^n,0$ is the germ of an analytic hypersurface in $\Bbb{C}^n$, then there exists a Zariski open subset $U$ of $\Bbb{P}^{n-1}$, such that condition $iL_A$ holds for the pair $\widetilde X-U,U$ along $U$.\end{theorem}

\begin{proof} We can view $\widetilde X$ locally as a family of hypersurfaces parameterized by $\Bbb{P}^{n-1}$. The fiber of the family over the plane $P$ is just the intersection $P\cap X$. The existence of $U$ follows from 4.3.\end{proof}

We can use the ideas of \cite{G-1} to describe these generic hyperplanes. We work in the chart $U_n$ given by planes $P$ with equation $z_n=\mathop{\sum}\limits_i a_{i} z_i$. Then we have local coordinates on  $E_{n-1}$ given by 
 $(z_1,...z_n,a_{1},...,a_{n-1})$. In these coordinates
we have $$\beta (z_1,...z_n,a_{1},...,a_{n-1})=(z_1,...z_n,\mathop{\sum}\limits_i a_{i} z_i)$$
If $\phi :{{\Bbb C}},0 \rightarrow \tilde X,P\times \{0\}$, then $\beta\circ\phi$ is tangent to P at the origin. If $\phi :{{\Bbb C}},0 \rightarrow  X, 0$ is tangent to $P$ at $0$, then $\phi$ lifts to $\tilde X,P\times \{0\}$, and we say $\phi$ is liftable. It follows from \cite{G-1}, that since 
  F defines $X$, $G:=F\circ \beta$ defines $\tilde X$. From the chain rule we note that 
$$\pd G{a_{i}}=z_i\pd F{z_n}\circ\beta,\hskip 2em
J_z(G)=(\pd F{z_j}\circ\beta+\mathop{\sum}\limits_ia_{i}\pd F{z_n}\circ\beta), 1\le j\le n-1.$$

\begin{cor}  Suppose $X^n,0$ is the germ of an analytic hypersurface in $\Bbb{C}^n$, then, for $P\in U_n$, $P$ is a point in the Z-open set of the last theorem, if and only if $z_i\pd F{z_n}\circ\beta\in (J_z(G))_S$ for $1\le i\le n-1$ at $P,0$. 

\end{cor}
\begin{proof} In the framework of the corollary, the condition of the corollary is exactly the $iL_A$ condition. \end{proof}

The corollary says that to check a plane is generic, it suffices to check that for all curves $\phi_i$ $i=1,2$ on $X$, tangent to $P$ at the origin, with lifts $\widetilde \phi_i$ for $\phi_i$, and $\Phi:= ( \phi_1,  \phi_2)$, $\widetilde\Phi:= (\widetilde \phi_1, \widetilde \phi_2)$, that $$(z_i\pd F{z_n})_D\circ \Phi\in ((\pd F{z_j})_D\circ\Phi+(\mathop{\sum}\limits_ia_{i}\pd F{z_n}\circ\beta)_D \circ \widetilde \Phi).$$
\vspace{.1 in}

We will give a description using analytic invariants of these generic hyperplanes.  For the rest of this section we will assume that the planes we consider are not limiting tangent hyperplanes to $X,0$. This condition is equivalent to ${\overline{J(F)_H}}={\overline {J(F)}}$ in $\cO_{X,0}$. 

The invariant we will use appeared earlier in section 3. It is the multiplicity of the pair $J(X\cap H)_D,{\overline {J(X\cap H)}}_D$, which we denote $e(J(X\cap H)_D,{\overline {J(X\cap H)}}_D)$.

Similar invariants have been used in this setting before. In the case of ICIS singularities, to test for whether or not a hyperplane is in the generic set of planes for which the hyperplane sections form a Whitney equisingular family, you use the multiplicity of the pair $(JM(X\cap H), \cO^p_{X})$, which is $e(JM(X\cap H))$. The plane is generic if this multiplicity is minimal, and the minimal number is the sum of the Milnor numbers of $X\cap H$, and $X\cap H\cap G$, where $H$ and $G$ are generic hyperplanes. 

The proof that the minimal value of $e(J(X\cap H)_D,{\overline {J(X\cap H)}}_D)$ again identifies generic hyperplanes will be done in the context of the multiplicity polar theorem, so we identify the modules we will use. 

We will work in $\widetilde X\times_{\Bbb P^{n-1}}\widetilde X\subset X\times {\Bbb P^{n-1}}\times X$. The module $N$ will be $(\beta^*{\overline {J(F)})}_D$, and the module $M$ will be $J_z(G)_D$. Notice that $M$ restricted to the fiber of the family over the plane $H$ is just $J(X\cap H)_D$, while $N$ restricted to $H$ is $(\overline {J(X)}|_H)_D$; because we are assuming $H$ is not a limiting tangent hyperplane, we have that $\overline {J(X)}|_H=\overline{J(X\cap H)}$, hence  $N$ restricted to $H$ is $\overline{J(X\cap H)}_D$, so the multiplicity of the pair $M(H), N(H)$ is the same as $e(J(X\cap H)_D,{\overline {J(X\cap H)}}_D)$. At this time we do not have a geometric interpretation of this number.

\begin{theorem} Suppose $X^{n-1},0$ is an isolated singularity hypersurface and $U$ the set of hyperplanes which are limiting tangent hyperplanes to $X$ at $0$. Then 

1) $e(J(X\cap H)_D,{\overline {J(X\cap H)}}_D)$ is upper semicontinuous on $U$.

2) The $iL_A$ condition holds along $U$ at a hyperplane $H$ for which the value of $e(J(X\cap H)_D,{\overline {J(X\cap H)}}_D)$ is minimal.
\end {theorem}

\begin{proof} The condition on $U$ implies that ${\overline {J(X\cap H)}}_D)$ is the restriction of $N$ to the fiber. Essentially since $N$ is independent of $H$, $N$ has no polar variety of the same codimension as $U$. The multiplicity polar theorem then implies 
$e(J(X\cap H)_D,{\overline {J(X\cap H)}}_D)$ is upper semicontinuous on $U$.

Suppose we are at $H$ which gives the minimal value of the multiplicity. Since the value of the multiplicity cannot go down, it must be constant, which implies that the polar variety of $M$ of the same dimension as $U$ must be empty. The emptiness of the polar variety puts restrictions on the size of the fiber of $Proj\cR(M)$. Now we know that generically the $\pd G{a_{i}}$ are in $\overline M$; coupling this with the bound on the dimension of the fiber of $Proj\cR(M)$, by Theorem A1 of \cite{KT1}, it follows that the $\pd G{a_{i}}$ are in the integral closure of $M$ at $H$ as well, which finishes the proof.\end{proof}

 \end{document}